\documentclass[11pt,english]{amsart}
\usepackage[T1]{fontenc}
\usepackage[latin1]{inputenc}
\usepackage{amstext}
\usepackage{amsthm}
\usepackage{amssymb}

\makeatletter
\numberwithin{equation}{section}
\numberwithin{figure}{section}
\theoremstyle{plain}
\newtheorem{thm}{\protect\theoremname}
\theoremstyle{remark}
\newtheorem{rem}[thm]{\protect\remarkname}
\theoremstyle{plain}
\newtheorem{lem}[thm]{\protect\lemmaname}
\theoremstyle{definition}
\newtheorem{example}[thm]{\protect\examplename}
\theoremstyle{plain}
\newtheorem{prop}[thm]{\protect\propositionname}
\theoremstyle{plain}
\newtheorem{cor}[thm]{\protect\corollaryname}
\theoremstyle{remark}
\newtheorem*{rem*}{\protect\remarkname}

\makeatother

\makeatother

\usepackage{babel}
\providecommand{\corollaryname}{Corollary}
\providecommand{\examplename}{Example}
\providecommand{\lemmaname}{Lemma}
\providecommand{\propositionname}{Proposition}
\providecommand{\remarkname}{Remark}
\providecommand{\theoremname}{Theorem}

\begin{document}

\title{Explicit Nikulin configurations on Kummer surfaces}

\addtolength{\textwidth}{3mm}
\addtolength{\hoffset}{3mm} 
\addtolength{\textheight}{5mm}
\addtolength{\voffset}{5mm} 


\global\long\def\CC{\mathbb{C}}%
 
\global\long\def\BB{\mathbb{B}}%
 
\global\long\def\PP{\mathbb{P}}%
 
\global\long\def\QQ{\mathbb{Q}}%
 
\global\long\def\RR{\mathbb{R}}%
 
\global\long\def\FF{\mathbb{F}}%

\global\long\def\DD{\mathbb{D}}%
 
\global\long\def\NN{\mathbb{N}}%
\global\long\def\ZZ{\mathbb{Z}}%
 
\global\long\def\HH{\mathbb{H}}%
 
\global\long\def\Gal{{\rm Gal}}%
\global\long\def\OO{\mathcal{O}}%
\global\long\def\pP{\mathfrak{p}}%

\global\long\def\pPP{\mathfrak{P}}%
 
\global\long\def\qQ{\mathfrak{q}}%
 
\global\long\def\mm{\mathcal{M}}%
 
\global\long\def\aaa{\mathfrak{a}}%
 
\global\long\def\a{\alpha}%
 
\global\long\def\b{\beta}%
 
\global\long\def\d{\delta}%
 
\global\long\def\D{\Delta}%
 
\global\long\def\L{\Lambda}%
 
\global\long\def\g{\gamma}%

\global\long\def\G{\Gamma}%
 
\global\long\def\d{\delta}%
 
\global\long\def\D{\Delta}%
 
\global\long\def\e{\varepsilon}%
 
\global\long\def\k{\kappa}%
 
\global\long\def\l{\lambda}%
 
\global\long\def\m{\mu}%
 
\global\long\def\o{\omega}%
 
\global\long\def\p{\pi}%
 
\global\long\def\P{\Pi}%
 
\global\long\def\s{\sigma}%

\global\long\def\S{\Sigma}%
 
\global\long\def\t{\theta}%
 
\global\long\def\T{\Theta}%
 
\global\long\def\f{\varphi}%
 
\global\long\def\deg{{\rm deg}}%
 
\global\long\def\det{{\rm det}}%

\global\long\def\Dem{Proof: }%
 
\global\long\def\ker{{\rm Ker\,}}%
 
\global\long\def\im{{\rm Im\,}}%
 
\global\long\def\rk{{\rm rk\,}}%
 
\global\long\def\car{{\rm car}}%
\global\long\def\fix{{\rm Fix( }}%

\global\long\def\card{{\rm Card\  }}%
 
\global\long\def\codim{{\rm codim\,}}%
 
\global\long\def\coker{{\rm Coker\,}}%
 
\global\long\def\mod{{\rm mod }}%

\global\long\def\pgcd{{\rm pgcd}}%
 
\global\long\def\ppcm{{\rm ppcm}}%
 
\global\long\def\la{\langle}%
 
\global\long\def\ra{\rangle}%

\global\long\def\Alb{{\rm Alb(}}%
 
\global\long\def\Jac{{\rm Jac(}}%
 
\global\long\def\Disc{{\rm Disc(}}%
 
\global\long\def\Tr{{\rm Tr(}}%
 
\global\long\def\NS{{\rm NS(}}%
 
\global\long\def\Pic{{\rm Pic(}}%
 
\global\long\def\Pr{{\rm Pr}}%

\global\long\def\Km{{\rm Km(}}%
\global\long\def\rk{{\rm rk(}}%
\global\long\def\Hom{{\rm Hom(}}%
 
\global\long\def\End{{\rm End(}}%
 
\global\long\def\aut{{\rm Aut}}%
 
\global\long\def\SSm{{\rm S}}%
\global\long\def\psl{{\rm PSL}}%

\global\long\def\diag{{\rm Diag}}%

\author{Xavier Roulleau, Alessandra Sarti}
\begin{abstract}
A Nikulin configuration is the data of $16$ disjoint smooth rational
curves on a K3 surface. According to results of Nikulin, the existence
of a Nikulin configuration means that the K3 surface is a Kummer surface,
moreover the abelian surface from the Kummer structure is determined
by the $16$ curves. A classical question of Shioda is about the existence
of non isomorphic Kummer structures on the same Kummer K3 surface.
The question was studied by several authors, and it was shown that
the number of non-isomorphic Kummer structures is finite, but no explicit
geometric construction of such structures was given. In the paper
\cite{RS}, we constructed explicitly non isomorphic Kummer structures
on some Kummer surfaces. In this paper we generalize the construction
to Kummer surfaces with a weaker restriction on the degree of the
polarization and we describe some cases where the previous construction
does not work. 
\end{abstract}

\maketitle

\section{Introduction}

A (projective, as always in this paper) Kummer surface is obtained
as the desingularization of the quotient of an abelian surface by
an involution with $16$ isolated fixed points. It is well known that
Kummer surfaces are K3 surfaces and that their Picard number is at
least $17$, the rank $17$ sub-group being generated by the $16$
rational curves in the resolution of the $16$ nodes and by the polarization.
In \cite{Nikulin}, Nikulin showed the converse, i.e. that a K3 surface
containing $16$ disjoint smooth rational curves, or $(-2)$-curves,
is the Kummer surface associated to an abelian surface. Let $X$ be
a K3 surface; we call a {\it Kummer structure} on $X$ an abelian
surface $A$ (up to isomorphism) such that $X\simeq Km(A)$, and we
call a {\it Nikulin configuration} a set of $16$ disjoint smooth
rational curves on $X$. By the result of Nikulin we have a bijection:
$$ \{\mbox{Kummer\,\,structures}\}
\longleftrightarrow \{\mbox{Nikulin\,\,configurations}\}/\mbox{Aut(X)} $$ 

In 1977, see \cite[Question 5]{Shioda}, T. Shioda raised the following
question : 

Is it possible to have non-isomorphic abelian surfaces $A$ and $B$,
such that $Km(A)$ and $Km(B)$ are isomorphic? 

Shioda and Mitani in \cite[Theorem 5.1]{MS} answer negatively the
question if $\rho(Km(A))=20$, where $\rho(Km(A))$ is the Picard
number of $Km(A)$, i.e. the rank of the Néron-Severi group of $Km(A)$.
The answer is also negative if $A$ is a generic principally polarized
abelian surface, i.e. $A$ is the jacobian of a curve of genus $2$
and $\rho(A)=1$. Then in \cite[Theorem 1.5]{GH}, Gritsenko and Hulek
answered positively the question. They showed that if $A$ is a generic
$(1,t)$-polarized abelian surface with $t>1$ then the abelian surface
$A$ and its dual $\hat{A}$, though not isomorphic, satisfy $Km(A)\cong Km(\hat{A})$.
In \cite[Theorem 0.1]{HLOY}, Hosono, Lian, Oguiso and Yau, by using
lattice theory, showed that the number of Kummer structures is finite
and for each integer $N\in\mathbb{N}^{*}$, they construct a Kummer
surface of Picard number $18$ with at least $N$ Kummer structures. 

In \cite[Example 4.16]{Orlov}, Orlov showed that if $A$ is a generic
abelian surface (i.e. $\rho(Km(A))=17$) then the number of abelian
surfaces (up to isomorphism) with equivalent bounded derived categories
is $2^{\nu}$, where $\nu$ is the number of prime divisors of $\frac{1}{2}M^{2}$,
for $M$ an ample generator of the Néron-Severi group of $A$. By
\cite[Theorem 0.1]{HLOY}, there is a one-to-one correspondence between
these equivalent bounded derived categories of $A$ and the Kummer
structures on the Kummer surface $\Km A)$ associated to $A$. Thus,
for example if $A$ is principally polarized we have that $M^{2}=2$
so that $\nu=0$ and we find again the fact that in this case there
is only one Kummer structure on $\Km A)$. Observe that $\nu$ can
be also defined as the number of prime divisors of $\frac{1}{4}L^{2}$,
where $L$ is the polarization induced by $M$ on $Km(A)$, (in particular
$L$ is orthogonal to the $16$ rational curves ; it is easy to see
that by changing the $16$ rational disjoint curves, the number $\nu$
does not change).

In \cite[Theorem 1]{RS}, we constructed explicit examples of two
Nikulin configurations $\mathcal{C}$, $\mathcal{C}'$ on some K3
surface $X$ such that the abelian surfaces $A$ and $A'$ associated
to these two configurations are not isomorphic. This was the first
geometric construction of two distinct Kummer structures. These examples
are for generic Kummer surfaces, such that the orthogonal complement
of the $16$ rational curves in $\mathcal{C}$ is generated by a class
$L$ such that $L^{2}=2k(k+1)$ for some integer $k$ (we give a motivation
for this restriction in the Appendix of this paper). 

The main goal of this paper is to provide a generalization of that
result to other Kummer surfaces. For that aim, let $t\in\NN^{*}$
be an integer and let $X$ be a general Kummer surface with a Nikulin
configuration $\mathcal{C}$ such that the orthogonal complement of
the $16$ $(-2)$-curves $A_{1},\dots,A_{16}$ in $\mathcal{C}$ is
generated by $L$ with $L^{2}=4t$. A class $C$ of the form $C=\b L-\a A_{1}$
with $\b\in\NN^{*}$ has self-intersection $C^{2}$ equals to $-2$
if and only if the coefficients $(\a,\b)$ satisfy the Pell-Fermat
equation $\a^{2}-2t\b^{2}=1.$ There is a non-trivial solution if
and only if $2t$ is not a square. Let us suppose that this is the
case. Then there exists a so-called fundamental solution which we
denote by $(\a_{0},\b_{0})$. Our main result is as follows:
\begin{thm}
\label{thm:MAIN}Suppose that $\b_{0}$ is even. Then $\b_{0}L-\a_{0}A_{1}$
is the class of an irreducible $(-2)$-curve $A_{1}'$, which curve
is disjoint from $A_{2},\dots,A_{16}$.\\
The Nikulin configurations $\mathcal{C}=\sum_{i=1}^{16}A_{i}$ and
$\mathcal{C}'=A_{1}'+\sum_{i=2}^{16}A_{i}$ define the same Kummer
structure on the Kummer surface $X$ if and only if the negative Pell-Fermat
equation $\a^{2}-2t\b^{2}=-1$ has a solution.\\
Suppose that this is the case. Then there exists a double cover map
$X\to\PP^{2}$ branched over $6$ lines $L_{1},\dots,L_{6}$, contracting
the $15$ $(-2)$-curves $A_{j},j\geq2$ to the singularities of $\sum_{i=1}^{6}L_{i}$,
and such that the induced involution exchanges the curves $A_{1},A_{1}'$
and therefore the configurations $\mathcal{C},\mathcal{C}'$. 
\end{thm}

The integers $t\in\NN$ such that $\b_{0}$ is even have density at
least $\frac{3}{4}$ ; among these integers, at least $\frac{2}{3}$
are such that the negative Pell-Fermat equation have no solution,
and therefore give examples of two distinct Kummer structures (see
Remark \ref{rem:DensityAtLeast1/2} for a precise meaning of that
affirmation, and also the table in the Appendix). As a by-product
of our study, let us mention the following result (see Proposition
\ref{prop:Degre4ModelWith15Nodes}), which we believe can be of independent
interest: \textit{Suppose that the equation $2\mu^{2}-t\nu^{2}=-1$
has a solution. Then there exists a model of the K3 surface $X$ as
a quartic surface in $\PP^{3}$ with $15$ nodes. }

One could also raise a weaker question than Shioda's question by asking
if $Km(A)\cong Km(B)$ implies $A$ and $B$ must be isogenous ? The
answer is positive and the result was surely known, but we could not
find an explicit proof in the literature, hence we recall it in Section
\ref{sec:Generalizations-for-other} and we show how it can be obtained
as a direct consequence of a result of Stellari \cite[Theorem 1.2]{Stellari}.
In the rest of the paper, we point out why the construction in Theorem
\ref{thm:MAIN} can not work for $\beta_{0}$ odd, moreover we study
examples of Nikulin configurations in the case that $\beta_{0}$ is
odd or $2t$ is a square.

\textbf{Acknowledgements:} We thank P. Stellari for pointing out his
paper \cite{Stellari}. We also thank K. Hulek, H. Lange, K. Oguiso,
M. Ramponi, J. Rivat and T. Shioda for useful discussions. We are
very grateful to the referee for the many questions, remarks and comments
that improved substantially this article. 

\section{\label{sec:Generalizations-for-other}Construction of Nikulin configurations}

\subsection{\label{subsec:Cases-when-not square}Preliminaries: the Pell-Fermat
equation and its negative}

The aim of this first sub-section is to recall results on Pell-Fermat
equations. We give various criteria when the fundamental solution
$(\a_{0},\b_{0})$ of it is such that $\b_{0}$ is even, and when
the negative Pell-Fermat solution has no solution.

\subsubsection{The Pell-Fermat equation}

For $t\in\NN^{*}$, the Pell-Fermat equation 
\begin{equation}
\a^{2}-2t\b^{2}=1\label{eq:Pell Fermat 2t ok}
\end{equation}
has a non-trivial solution $(\a,\b)\in\ZZ^{2}$ if and only if $2t$
is not a square. Then there exists a fundamental solution $(\a_{0},\b_{0})\in\NN$,
such that for every other solution $(\a,\b)$, there exists $k\in\ZZ$
with $\a+\b\sqrt{2t}=\pm(\a_{0}+\sqrt{2t}\b_{0})^{k}$. 
\begin{rem}
\label{rem:pell-fermat-tricks}For $k\in\ZZ$, let $(x_{k},y_{k})\in\ZZ^{2}$
be such that 
\[
x_{k}+y_{k}\sqrt{2t}=(\a_{0}+\sqrt{2t}\b_{0})^{k}.
\]
Using that $\a_{0}\geq1,\b_{0}\geq1$ and an induction, one can check
that the sequences $(x_{k})_{k\in\NN},\,(y_{k})_{k\in\NN}$ are strictly
increasing with $k\in\NN$. Therefore, if $(\a,\b)\in\NN^{2}$ is
a solution different from $(1,0)$ and with $\a\leq\alpha_{0}$ or
$\b\leq\b_{0}$, then $(\a,\b)$ is the fundamental solution. \\
We observe moreover that for a solution $(\a,\b)$ of equation \eqref{eq:Pell Fermat 2t ok},
the integer $\a$ is necessarily odd.
\end{rem}

For $t$ a positive integer such that $2t$ is not a square, we denote
by $(\a_{0},\b_{0})$ the fundamental solution of $\a^{2}-2t\b^{2}=1$.
Part a) of the following Lemma shows that the density of integers
$t$ such that $\b_{0}$ is even is at least $\frac{3}{4}$ (we thank
Joël Rivat for useful discussions on that question, and also Lemma
\ref{lem:Rivat2} below):
\begin{lem}
\label{lem-Rivat}a) Suppose that $t\neq0\text{ mod }4$. Then $\b_{0}$
is even.\\
b) There is an infinite number of integers $s$ such that  the fundamental
solution $(\a_{0},\b_{0})$ of $\a^{2}-8s^{2}\b^{2}=1$ has odd $\b_{0}$.
\\
c) There is an infinite number of integers $s$ such that the fundamental
solution $(\a_{0},\b_{0})$ of $\a^{2}-8s^{2}\b^{2}=1$ has even $\b_{0}$. 
\end{lem}

\begin{proof}
Let $(\a,\b)$ be a solution of equation $\a^{2}-2t\b^{2}=1$. Suppose
that $\b$ is odd. Then 
\[
\b=\pm1,\pm3\text{ mod }8,
\]
and one has $\b^{2}=1\,\text{mod}\,8$. Since $\a^{2}-2t\b^{2}=1$,
one has $\a^{2}=1+2t$ mod $8$. Since $\a$ is also odd, $\a^{2}=1\,\text{mod }8$,
thus $2t=0\text{ mod }8$ and therefore $t=0\text{ mod 4}$. That
proves part a).

Let $(x_{1},y_{1})$ be the fundamental solution of $x^{2}-2ty^{2}=1$.
For $n\in\ZZ$, the integers $\pm x_{n},\pm y_{n}$ defined by 
\[
x_{n}+y_{n}\sqrt{2t}=(x_{1}+y_{1}\sqrt{2t})^{n}
\]
are the solutions of equation $x^{2}-2ty^{2}=1$. The sequence $(y_{n})_{n\geq1}$
is strictly increasing and we see that the fundamental solution of
\[
x^{2}-2ty_{n}^{2}y^{2}=1
\]
is $(x_{n},1)$. Using part a), we remark that always $ty_{n}^{2}=0\text{ mod 4}$.
Take now $t=4$, we therefore obtain result b). For $n$ even, $y_{n}$
is even; let $z_{n}$ be such that $y_{n}=2z_{n}$. The fundamental
solution of 
\[
x^{2}-2tz_{n}^{2}y^{2}=1
\]
is $(x_{n},2)$ ; taking $t=4$ as in the previous case, one obtains
result c).
\end{proof}
\begin{example}
For $1\leq s\leq100$ such that $8s$ is not a square, i.e. for $s\not\in\{2,8,18,32,50,72,98\}$,
the fundamental solution $(\a_{0},\b_{0})$ of equation $\a^{2}-8s\b^{2}=1$
is such that $\b_{0}$ is even if and only if $s$ is in 
\[
\{7,9,14,23,30,31,33,34,46,47,56,57,62,63,69,71,73,75,77,79,81,82,89,90,94\}.
\]
\end{example}

\subsubsection{The negative Pell-Fermat equation}

The equation 
\begin{equation}
\a^{2}-2t\b^{2}=-1\label{eq:Neg-Pell_fermat}
\end{equation}
is called the \textit{negative Pell-Fermat equation}. If $(x,y)$
is a solution, then $(\a,\b)=(x^{2}-2ty^{2},2xy)$ is a solution of
the Pell-Fermat equation \eqref{eq:Pell Fermat 2t ok}, with $\b$
even. The negative Pell-Fermat equation can be solved by the method
of continued fractions and it has solutions if and only if the period
of the continued fraction has odd length. A necessary (but not sufficient)
condition for solvability is that $t$ is not divisible by a prime
of form $4k+3$. The following Lemma implies that the density of integers
$t$ such that the negative Pell-Fermat equation \eqref{eq:Neg-Pell_fermat}
has no solution is at least $\frac{5}{6}$:
\begin{lem}
\label{lem:Rivat2}Suppose that the negative Pell-Fermat equation
\eqref{eq:Neg-Pell_fermat} has a solution. Then $t=1\,mod\,4$ and
$t\neq0\,mod\,3$, in other words: $t=1$ or $5\,mod\,12$. 
\end{lem}

\begin{proof}
Suppose that $(\a,\b)$ is a solution of equation \eqref{eq:Neg-Pell_fermat}.
Since $\a^{2}-2t\b^{2}=-1$, the integer $\a$ is odd, thus $\a^{2}=1\,\text{mod}\,8$,
and $2t\b^{2}=\a^{2}+1=2\,\text{mod}\,8$, which implies that $\b$
is odd (otherwise $2t\b^{2}=0\,\text{mod}\,8$), thus $\b^{2}=1\,\text{mod}\,8$.
In that way $2t=2\,\text{mod}\,8$ hence $t=1\,\text{mod}\,4$.\\
Since $\a^{2}=0$ or $1\,\text{mod}\,3$, one has $2tb^{2}=1$ or
$2\,\text{mod}\,3$, thus $t\neq0\,\text{mod}\,3$. 
\end{proof}
The first few numbers $t$ for which equation \eqref{eq:Neg-Pell_fermat}
is solvable are
\[
1,5,13,25,29,37,41,53,61,65,85,101,109...
\]

\begin{rem}
\label{rem:DensityAtLeast1/2}From Lemmas \ref{lem-Rivat} and \ref{lem:Rivat2},
we conclude that the density of integers $t$ such that the negative
Pell-Fermat equation \eqref{eq:Neg-Pell_fermat} has no solution and
the Pell-Fermat equation \eqref{eq:Pell Fermat 2t ok} has a solution
$(\a_{0},\b_{0})$ with $\b_{0}$ even is at least $\frac{7}{12}$.
\end{rem}

\subsection{The general problem\label{subsec:The-general-problem}}

\subsubsection{Isogenies}

Before to state our results about the question of Shioda \cite[Question 5]{Shioda}
recalled in the Introduction, we can generalize the problem to the
following question:

\textit{Given two abelian surfaces $A$ and $B$ such that $Km(A)\cong Km(B)$
are then $A$ and $B$ isogenous ?}

The answer is positive and certainly well known, in particular to
people working on derived categories on abelian surfaces. For convenience
we give here a short proof: 
\begin{prop}
Let $A$ and $B$ be abelian surfaces such that the associated Kummer
surfaces are isomorphic, then $A$ and $B$ are isogenous abelian
surfaces. 
\end{prop}

\begin{proof}
Since $Km(A)\cong Km(B)$ then the derived categories $D^{b}(Km(A))$
and $D^{b}(Km(B))$ are equivalent. Thus by \cite[Theorem 1.2]{Stellari},
the abelian surfaces are isogenous. 
\end{proof}

\subsubsection{Notations and known results on the Néron-Severi group of a Kummer
surface.}

Let $t\in\NN$ be an integer and let $B$ be a generic Abelian surface
with (primitive) polarization $M$ such that $M^{2}=2t$. Let $X=\Km B)$
be the associated Kummer surface. Let $A_{1},\dots,A_{16}$ be the
$16$ disjoint $(-2)$-curves on $X$ that are resolution of the singularities
of the quotient $B/[-1]$. By \cite[Proposition 3.2]{Morrison}, \cite[Proposition 2.6]{GS},
corresponding to the polarization $M$ on $B$, there is a primitive
big and nef divisor $L$ on $\Km B)$ such that 
\[
L^{2}=4t
\]
and $LA_{i}=0,\,i\in\{1,\dots,16\}$. The Néron-Severi group of $X=\Km B)$
satisfies: 
\[
\ZZ L\oplus K\subset\NS X),
\]
where $K$ denotes the Kummer lattice (the saturated lattice containing
the $16$ disjoint $(-2)$-curves $A_{i},\,i=1,\dots,16$) which is
a negative definite lattice of rank $16$ and discriminant $2^{6}$.
For $B$ generic among polarized Abelian surfaces $\rk\NS X))=17$
and $\NS X)$ is an over-lattice of index two of $\ZZ L\oplus K$
which is described precisely in \cite[Theorem 2.7]{GS}, in particular
we will repeatedly use the following result:
\begin{lem}
\label{lem:At most 4 beta}(\cite[Remarks 2.3 \& 2.10]{GS}) An element
$\Gamma\in\NS X)$ has the form $\Gamma=\alpha L-\sum\beta_{i}A_{i}$
with $\alpha,\beta_{i}\in\frac{1}{2}\ZZ$. If $\alpha$ or $\beta_{i}$
for some $i$ is in $\frac{1}{2}\ZZ\setminus\ZZ,$ then at least $4$
of the $\beta_{j}$'s are in $\frac{1}{2}\ZZ\setminus\ZZ$. If $\alpha\in\ZZ$,
then at least $8$ of the $\beta_{j}$'s are in $\frac{1}{2}\ZZ\setminus\ZZ$
or $\forall j,\,\b_{j}\in\ZZ$.
\end{lem}

\subsubsection{The Pell-Fermat equation and construction of $(-2)$-classes\label{subsec:The-3-equations}}

We are looking for a polarization $L'$ and a class $A_{1}'$ of the
form 
\[
\begin{array}{c}
A'_{1}=\b L-\a A_{1}\\
L'=bL-aA_{1}
\end{array},
\]
with $\a,\b,a,b\in\NN\setminus\{0\}$ such that one has $A_{1}'^{2}=-2$,
$L'A_{1}'=0$ and $L'^{2}=L^{2}=4t$. These three conditions are respectively
\begin{equation}
\begin{array}{c}
\a^{2}-2t\b^{2}=1\\
2tb\b=a\a\\
a^{2}=2t(b^{2}-1)
\end{array},\label{eq:tree conditions}
\end{equation}
the first expresses that $A_{1}'$ is a $(-2)$-class, the second
that this $(-2)$-class is disjoint from the polarisation $L'$, the
third that $L^{2}=L'^{2}$. We will use the divisor $L'$ and the
property that $L'A_{1}'=0$ in order to show that $A_{1}'$ can be
represented by an irreducible curve. 
\begin{lem}
There are non-trivial solutions to the three equations \eqref{eq:tree conditions}
if and only if $2t$ is not a square. In that case, if $(\a,\b)$
is a solution of the first equation in \eqref{eq:tree conditions},
one has 
\[
(a,b)=(2t\b,\a).
\]
\end{lem}

\begin{proof}
In order that the Pell-Fermat equation \eqref{eq:Pell Fermat 2t ok}
admits a solution, we need that $2t$ is not a square. Let us suppose
that this is the case and let $(\a,\b)$ be such a solution, which
we can suppose with $\a>0,\,\b>0$. By replacing $a=2t\frac{\b}{\a}b$
in the third equation, one gets 
\[
4t^{2}b^{2}\b^{2}=2t\a^{2}(b^{2}-1),
\]
which is equivalent to 
\[
b^{2}(\a^{2}-2t\b^{2})=\a^{2},
\]
since $\a^{2}-2t\b^{2}=1$ and we search solutions with $b>0$, we
obtain $b=\a$. Then by the third equality, we get $a^{2}=2t(\a^{2}-1)$
and equality $\a^{2}-1=2t\b^{2}$ implies $a=2t\b$.
\end{proof}

\subsection{\label{subsec:The--odd}The $\protect\b_{0}$ odd case}

Suppose $2t$ is not a square and let $(\a_{0},\b_{0})$ be a solution
of equation \eqref{eq:Pell Fermat 2t ok}. Let us suppose that $\b_{0}$
is odd and let us define 
\[
A_{1}'=\b_{0}L-\a_{0}A_{1},
\]
which is a $(-2)$-class. Then 
\begin{prop}
The $(-2)$-class $A_{1}'=\b_{0}L-\a_{0}A_{1}$ cannot be the class
of an irreducible rational curve. 
\end{prop}

\begin{proof}
Suppose that $A_{1}'$ is irreducible. Then we have two Nikulin configurations
\[
\mathcal{C}=\sum_{i=1}^{16}A_{i},\,\mathcal{C}'=A_{1}'+\sum_{i=2}^{16}A_{i}.
\]
Since Nikulin configurations are $2$-divisible (see \cite{Nikulin}),
the divisor $A_{1}+A_{1}'$ is $2$-divisible and 
\[
\frac{1}{2}(A_{1}+A_{1}')=\frac{\b_{0}}{2}L-\frac{\a_{0}-1}{2}A_{1}
\]
is an integral class. Since $\b_{0}$ is odd by assumption and $\a_{0}$
must be odd by the equality $\a_{0}^{2}-2t\b_{0}^{2}=1$, it follows
that $\frac{L}{2}\in\NS X)$, which contradicts $L$ being a primitive
class.
\end{proof}
We will come back to this case (when $\b_{0}$ is odd) in Subsection
\ref{subsec:An-example-whent odd}.

\subsection{\label{subsec:The--even}The $\protect\b_{0}$ even case: $\protect\b_{0}L-\protect\a_{0}A_{1}$
is the class of a $(-2)$-curve}

Assume $2t$ is not a square. Let $(\a_{0},\b_{0})$ be the fundamental
solution of the Pell-Fermat equation \eqref{eq:Pell Fermat 2t ok}.
We assume in this section that $\beta_{0}$ is even and we define
as in Subsection \ref{subsec:The-3-equations} the classes:
\[
A_{1}'=\b_{0}L-\a_{0}A_{1},\,\,L'=\a_{0}L-2t\b_{0}A_{1}.
\]
One has $A_{1}'^{2}=-2$, $L'A_{1}'=0,$ $L'^{2}=L^{2}=4t$. 
\begin{prop}
Suppose that $\b_{0}$ is even. The class $L'$ is big and nef and
the classes $A_{1}',A_{2}\dots,A_{16}$ are the only $(-2)$-classes
contracted by $L'$.
\end{prop}

\begin{proof}
The class $L'$ is nef if and only if for any $(-2)$-curves $\G$,
one has $\G L'\geq0$. Let 
\[
\G=uL-\sum_{i=1}^{16}v_{i}A_{i}
\]
be a $(-2)$-curve (thus $\sum v_{i}^{2}-2tu^{2}=1$) ; we recall
that by Lemma \ref{lem:At most 4 beta}, if one coefficient $u$ or
$v_{i}$ is in $\frac{1}{2}\ZZ\setminus\ZZ$, then at least four of
the $v_{i}$'s are in $\frac{1}{2}\ZZ\setminus\ZZ$. Suppose that
\[
\G L'\leq0,
\]
this is equivalent to 
\[
u\a_{0}\leq v_{1}\b_{0},
\]
in other words $u\leq\frac{\b_{0}}{\a_{0}}v_{1}$, thus 
\[
\sum_{i\geq1}v_{i}^{2}=2tu^{2}+1\leq2t\left(\frac{\b_{0}^{2}}{\a_{0}^{2}}v_{1}^{2}\right)+1=2t\b_{0}^{2}\left(\frac{v_{1}^{2}}{\a_{0}^{2}}\right)+1
\]
 and therefore using the relation $\a_{0}^{2}-2t\b_{0}^{2}=1$, one
obtains 
\[
\sum_{i\geq1}v_{i}^{2}\leq(\a_{0}^{2}-1)\left(\frac{v_{1}^{2}}{\a_{0}^{2}}\right)+1
\]
and therefore
\[
\sum_{i\geq2}v_{i}^{2}\leq1-\frac{v_{1}^{2}}{\a_{0}^{2}}.
\]
Apart from the trivial cases of curves $\G=A_{i}$ for $i>1$, one
can suppose $u>0$. If $v_{1}>\frac{1}{2}\a_{0}$ then $\sum_{i\geq2}v_{i}^{2}<\frac{3}{4}$,
then by Lemma \ref{lem:At most 4 beta}, one gets $v_{2}=\dots=v_{16}=0$,
$v_{1}=\a_{0}$, $u=\b_{0}$, so that $\G=A_{1}'$, for which $\G L'=0$.
Thus one can suppose that 
\begin{equation}
0<v_{1}\leq\frac{1}{2}\a_{0}\label{eq:truc23}
\end{equation}
(if $v_{1}=0$, then $u=0$, which we already excluded) and, up to
permutation of the indices: $v_{2}=v_{3}=v_{4}=\frac{1}{2}$ (since
$\sum_{i\geq2}v_{i}^{2}<1$ and by the structure of the Néron-Severi
group as described in Lemma \ref{lem:At most 4 beta}). The relation
$\sum v_{i}^{2}-2tu^{2}=1$ is now $v_{1}^{2}-2tu^{2}=\frac{1}{4}$,
which is 
\[
(2v_{1})^{2}-2t(2u)^{2}=1.
\]
Defining $V=2v_{1}\in\NN$ and $U=2u\in\NN$, we see that $(U,V)$
is a solution of the Pell-Fermat equation $\b^{2}-2t\a^{2}=1$. Moreover,
by Equation \eqref{eq:truc23} we know that $0<V\leq\a_{0}$. Since
by hypothesis $(\a_{0},\b_{0})$ is the primitive solution, and $V>0$,
we have $\a_{0}\leq V$. Therefore, by remark \ref{rem:pell-fermat-tricks}:
$V=\a_{0}$, which implies that $U=\b_{0}$, thus $v_{1}=\frac{1}{2}\a_{0}$,
$u=\frac{1}{2}\b_{0}$, and thus (for $\G\neq A_{2},\dots,A_{16}$)
we have 
\[
\G L'\leq0
\]
if and only if $\G L'=0$ and $\Gamma$ has the form $\G=\frac{1}{2}(\b_{0}L-\a_{0}A_{1}-A_{2}-A_{3}-A_{4})$.
But by Lemma \ref{lem:At most 4 beta}, in order for $\G$ to be in
$\NS X)$, the integer $\b_{0}$ must be odd, which is impossible
by our assumption on $\beta_{0}$. 

In conclusion, we obtain that $L'$ is big and nef, and if $\beta_{0}$
is even, then the only $(-2)$-classes $\G$ such that $\G L'=0$
are $A_{1}',A_{2},\dots,A_{16}$.
\end{proof}
Let us prove the following result:
\begin{prop}
Suppose that $\b_{0}$ is even. The line bundle $3L'$ (where $L'=\a_{0}L-2t\b_{0}A_{1}$)
defines a morphism $\phi_{3L'}:X\to\PP^{N}$ which is birational onto
its image and contracts exactly the divisor $A_{1}'=\b_{0}L-\a_{0}A_{1}$
and the $15$ $(-2)$-curves $A_{i},\,i\geq2$.
\end{prop}

\begin{proof}
By \cite[Section 3.8]{reid} either $|3L'|$ has no fixed part or
$3L'=aE+\Gamma$, where $|E|$ is a free pencil, and $\Gamma$ is
a $(-2)$-curve with $E\Gamma=1$. However if $E\G=1$, then $3L'E=aE^{2}+1$,
but since $E^{2}=0$ this is impossible. Thus $|3L'|$ has no fixed
part; moreover by \cite[Corollary 3.2]{SD}, it has then no base points.

Let us prove that the morphism $\phi_{3L'}$ has degree one, i.e.
that $|3L'|$ is not hyperelliptic (see \cite[Section 4]{SD}). By
loc. cit., $|3L'|$ is hyperelliptic only if there exists a genus
$2$ curve $C$ such that $3L'=2C$ or there exists an elliptic curve
$E$ such that $(3L')E=2$. Suppose we are in the first case. Since
$C^{2}=2$, one has $9\cdot4t=8$, which is impossible. The second
alternative is also readily impossible. Thus the morphism $\phi_{3L'}$
has degree one. Moreover since $3L'A_{1}'=3L'A_{2}=\dots=3L'A_{16}=0$,
the $16$ divisors are contracted by $\phi_{3L'}$.
\end{proof}
We obtain:
\begin{cor}
Suppose that $\b_{0}$ is even. The divisor $A_{1}'$ is an irreducible
$(-2)$-curve.
\end{cor}

\begin{proof}
Since $A_{1}'^{2}=-2$ and $LA_{1}'\geq0$, by Riemann-Roch Theorem
we can assume it is effective. Let $B$ be one of the divisors $A_{1}',A_{2},\dots,A_{16}$.
One has $3L'B=0$, thus the linear system $|3L'|$ contracts $B$
to a singular point. Since the Picard number of the K3 surface $X=\Km B)$
is $17$, that singularity must be a node and therefore $A_{1}'$
is irreducible.
\end{proof}

\subsection{\label{subsec:Two-Kummer-structures}Two Kummer structures in case
$\protect\b_{0}$ even and the negative Pell-Fermat equation is not
solvable}

Suppose that $2t$ is not a square and let $(\a_{0},\b_{0})$ be the
fundamental solution of the Pell-Fermat equation 
\begin{equation}
\l^{2}-2t\mu^{2}=1.\label{eq:Pell-Fermat-1}
\end{equation}
We suppose that $\b_{0}$ is even, and we recall that $A_{1}'=\b_{0}L-\a_{0}A_{1}$
in $\NS X)$. The aim of this Section is to prove the following:
\begin{thm}
\label{thm:no automorphisms}Suppose that $t\geq2$ and the negative
Pell-Fermat equation
\begin{equation}
\l^{2}-2t\mu^{2}=-1\label{eq:NegativePell-Fermat}
\end{equation}
has no solution. There is no automorphism $f$ of $X$ sending the
configuration $\mathcal{C}=\sum_{i=1}^{16}A_{i}$ to the configuration
$\mathcal{C}'=A_{1}'+\sum_{i=2}^{16}A_{i}$. 
\end{thm}

\begin{rem*}
We recall that the Nikulin configurations $\mathcal{C}$ and $\mathcal{C}'$
define two distinct Kummer structures if and only if there is no automorphism
sending $\mathcal{C}$ to $\mathcal{C}'$ (see \cite{HLOY}; a proof
is given in \cite[Proposition 21]{RS}). 
\end{rem*}
In order to prove Theorem \ref{thm:no automorphisms}, let us suppose
that such an automorphism $f$ exists. The group of translations by
the $2$-torsion points on $B$ acts on $X=\Km B)$ and that action
is transitive on the set of curves $A_{1},\dots,A_{16}$. Thus, up
to changing $f$ by $f\circ t$ (where $t$ is such a translation),
one can suppose that the image of $A_{1}$ is $A_{1}'$. Then the
automorphism $f$ induces a permutation of the curves $A_{2},\dots,A_{16}$.
The $(-2)$-curve $A_{1}''=f^{2}(A_{1})=f(A_{1}')$ is orthogonal
to the $15$ curves $A_{i},\,i>1$ and therefore its class is in the
group generated by $L$ and $A_{1}$. Let $\l,\mu\in\ZZ$ such that
$A_{1}''=\l A_{1}+\mu L$, so that $(\l,\m)$ is a solution of the
Pell-Fermat equation \eqref{eq:Pell-Fermat-1}. Let us prove:
\begin{lem}
\label{lem:treize}Let $C=\l A_{1}+\mu L$ be an effective $(-2)$-class.
Then there exists $u,v\in\NN$ such that $C=uA_{1}+vA_{1}'$, in particular
the only $(-2)$-curves in the lattice generated by $L$ and $A_{1}$
are $A_{1}$ and $A_{1}'$.
\end{lem}

\begin{proof}
If $(\l,\mu)$ is a solution of equation \eqref{eq:Pell-Fermat-1},
then so are $(\pm\l,\pm\mu)$. We say that a solution is positive
if $\l\geq0$ and $\mu\geq0$. Let us identify $\ZZ^{2}$ with $\ZZ[\sqrt{2t}]$
by sending $(\l,\mu)$ to $\l+\mu\sqrt{2t}$. The solutions of equation
\ref{eq:Pell-Fermat-1} are units of the ring $\ZZ[\sqrt{2t}]$. Let
$\a_{0}+\b_{0}\sqrt{2t}$ ($\a_{0},\b_{0}\in\NN^{*}$) be the fundamental
solution to equation \eqref{eq:Pell-Fermat-1}. The solutions with
positive coefficients are the elements of the form
\[
\l_{m}+\mu_{m}\sqrt{2t}=(\a_{0}+\b_{0}\sqrt{2t})^{m},\,m\in\NN.
\]
An effective $(-2)$-class $C=\l A_{1}+\mu L$ either equals $A_{1}$
or satisfies $CL>0$ and $CA_{1}>0$, therefore $\mu>0$ and $\l<0$.
Thus if $C\neq A_{1}$, there exists $m$ such that $C=-\l_{m}A_{1}+\mu_{m}L$.
Since $A_{1}'=\b_{0}L-\a_{0}A_{1}$ corresponds to the fundamental
solution of equation \eqref{eq:Pell-Fermat-1}, we have $L=\frac{1}{\b_{0}}(A_{1}'+\a_{0}A_{1})$
and we obtain
\[
C=-\l_{m}A_{1}+\frac{\mu_{m}}{\b_{0}}(A_{1}'+\a_{0}A_{1})=\frac{\mu_{m}}{\b_{0}}A_{1}'+(\frac{\a_{0}}{\b_{0}}b_{m}-\l_{m})A_{1}
\]
and the Lemma is proved if the coefficients $u_{m}=\frac{\mu_{m}}{\b_{0}}$
and $v_{m}=\frac{\a_{0}}{\b_{0}}\mu_{m}-\l_{m}$ are both positive
and in $\ZZ$. Using the fact that 
\[
\l_{m+1}+\mu_{m+1}\sqrt{2t}=(\a_{0}+\sqrt{2t}\b_{0})(\l_{m}+\mu_{m}\sqrt{2t}),
\]
we obtain 
\[
\begin{array}{c}
\l_{m+1}=\a_{0}\l_{m}+2t\b_{0}\mu_{m}\\
\mu_{m+1}=\a_{0}\mu_{m}+\b_{0}\l_{m}
\end{array}.
\]
Then we compute that
\[
u_{m+1}=\frac{\mu_{m+1}}{\b_{0}}=\a_{0}\frac{\mu_{m}}{\b_{0}}+\l_{m},\,\,v_{m+1}=\frac{\a_{0}}{\b_{0}}\mu_{m+1}-\l_{m+1}=\frac{\mu_{m}}{\b_{0}}
\]
and by induction we conclude that $u_{m},v_{m}$ are in $\NN$ for
any $m\geq1$.
\end{proof}
Lemma \ref{lem:treize} implies that $A_{1}''=A_{1}$ i.e. $f$ permutes
$A_{1}$ and $A_{1}'$. Let us continue the proof of Theorem \ref{thm:no automorphisms}:

Since the automorphism $f$ preserves the set 
\[
B=\{A_{1}',A_{1},\dots,A_{16}\},
\]
it acts with finite order $n_{0}$ on $B$. Since $B$ is a $\QQ$-basis
of $\NS X)\otimes\QQ$, the automorphism $f^{n_{0}}$ acts trivially
on $\NS X),$ thus it preserves an ample class, and by \cite[Proposition 5.3.3]{Huyb},
the automorphism $f^{n_{0}}$ has finite order, which proves that
$f$ has finite order. Up to taking an odd power of $f$, one can
suppose that $f$ has order $2^{m}$ for some $m\in\NN^{*}$. Suppose
$m=1$, i.e. $f$ is an involution. Then the integral class
\[
D=\frac{1}{2}(A_{1}+A_{1}')=\frac{\b_{0}}{2}L-\frac{\a_{0}-1}{2}A_{1}
\]
(recall that $\beta_{0}$ is even and $\alpha_{0}$ is odd) is fixed
by $f$. Let us define 
\[
d_{0}=GCD(\b_{0},\a_{0}-1)/2,
\]
then the class $D'=\frac{1}{d_{0}}D$ is primitive in $\NS X)$. We
have 
\[
D'^{2}=\frac{\a_{0}-1}{d_{0}^{2}}\in\ZZ,
\]
and in fact, since $\NS X)$ is an even lattice, $D'^{2}$ is even,
so that $\frac{\alpha_{0}-1}{2d_{0}^{2}}\in\mathbb{Z}$. Let us define
\[
W=\frac{2d_{0}^{2}}{\a_{0}-1}D'.
\]
Let $\NS X)^{f}$ be the sub-lattice of $\NS X)$ fixed by $f$. By
Lemma \ref{lem:At most 4 beta}, for any class $E$ in $\NS X)^{f}$,
there exists $a,b_{2},\dots,b_{16}\in\ZZ$ such that 
\[
E=\frac{1}{2}(aD'+\sum_{i=2}^{16}b_{i}A_{i}).
\]
Since $WD'=2$, we get $WE=a\in\ZZ$, therefore $W$ is an element
of the dual of $\NS X)^{f}$, and the discriminant group of $\NS X)^{f}$
contains the sub-group isomorphic to $\mathbb{Z}/\frac{\alpha_{0}-1}{2d_{0}^{2}}\mathbb{Z}$
generated by the class of $W$. 

\subsection*{Case when $f$ is a non-symplectic involution}

Suppose that $f$ is non-symplectic. Then (see e.g. \cite{AST}) $\NS X)^{f}$
is a $2$-elementary lattice, which means that the discriminant group
of $\NS X)^{f}$ is isomorphic to $(\ZZ/2\ZZ)^{h}$ for some positive
integer $h$. Since $\mathbb{Z}/\frac{\alpha_{0}-1}{2d_{0}^{2}}\mathbb{Z}$
is a sub-group of the discriminant group, and $f$ is supposed to
be non-symplectic, we get two cases:
\[
\frac{\alpha_{0}-1}{2d_{0}^{2}}\in\{1,2\}.
\]

\textit{Sub-case $\frac{\alpha_{0}-1}{2d_{0}^{2}}=1$.} If $\frac{\alpha_{0}-1}{2d_{0}^{2}}=1$,
then $\a_{0}=1+2d_{0}^{2}$. Since $d_{0}|\b_{0}/2$, there exists
$e_{0}\in\ZZ$ such that $\b_{0}=2e_{0}d_{0}$. From the relation
$\a_{0}^{2}-2t\b_{0}^{2}=1$, it follows that $d_{0}$ and $e_{0}$
satisfy the negative Pell-Fermat equation
\[
d_{0}^{2}-2te_{0}^{2}=-1.
\]
Conversely, suppose that $(d_{0},e_{0})$ is the primitive solution
to the equation $d^{2}-2te^{2}=-1$. The fundamental solution to the
Pell-Fermat equation is 
\[
\a_{0}+\b_{0}\sqrt{2t}=(d_{0}+e_{0}\sqrt{2t})^{2}=2d_{0}^{2}+1-2d_{0}e_{0}\sqrt{2t}
\]
 and $\frac{\alpha_{0}-1}{2d_{0}^{2}}=1$. Since, in the hypothesis
of Theorem \ref{thm:no automorphisms}, we supposed that the negative
Pell-Fermat equation has no solution, the case $\frac{\alpha_{0}-1}{2d_{0}^{2}}=1$
is excluded. 
\begin{rem}
In \cite{RS}, we studied the cases with $2t=k(k+1)$. Then $D=2D'=2L-2kA_{1}$,
which gives $d_{0}=1$ and $\a_{0}=2k+1$. The proof of \cite[Theorem 19]{RS}
implies that the negative Pell-Fermat equation $x^{2}-k(k+1)y^{2}=-1$
has no solution for $k\geq2$. 
\end{rem}

\textit{Sub-case $\frac{\alpha_{0}-1}{2d_{0}^{2}}=2$.} If $\frac{\alpha_{0}-1}{2d_{0}^{2}}=2$,
then $\a_{0}=1+4d_{0}^{2}$. Since $d_{0}|\b_{0}/2$, there exists
$e_{0}\in\ZZ$ such that $\b_{0}=2e_{0}d_{0}$. From the relation
$\a_{0}^{2}-2t\b_{0}^{2}=1$, it follows that $d_{0}$ and $e_{0}$
satisfy the relation
\begin{equation}
2d_{0}^{2}-te_{0}^{2}=-1,\label{eq:Relationd0e0}
\end{equation}
(conversely, if $(d_{0},e_{0})$ is a solution of \eqref{eq:Relationd0e0},
then $\a_{0}=1+4d_{0}^{2}$ and $\b_{0}=2e_{0}d_{0}$ is a solution
of the Pell-Fermat equation \eqref{eq:Pell-Fermat-1}). We have 
\[
D'=\frac{1}{d_{0}}D=e_{0}L-2d_{0}A_{1},
\]
with $D'^{2}=4,D'A_{1}=D'A_{1}'=4d_{0},D'A_{j}=0$ for $j\in\{2,\dots,16\}$.
Let us prove that 
\begin{prop}
\label{prop:Degre4ModelWith15Nodes} The divisor $D'$ is nef, the
linear system $|D'|$ is base point free, non hyperelliptic and defines
a morphism $\varphi:X\to\PP^{3}$ such that $\varphi(X)=Y$ is a quartic
surface with $15$ nodes, which are images of the disjoint curves
$A_{j},j\geq2$. 
\end{prop}

\begin{proof}
\textit{Let us prove that $D'$ is nef}. Let $\G=\a L-\sum_{i=1}^{16}\b_{i}A_{i}$
be a $(-2)$-curve: 
\begin{equation}
2t\a^{2}-\sum_{i=1}^{16}\b_{i}^{2}=-1\label{eq:minus2}
\end{equation}
where $\a,\b_{i}\in\frac{1}{2}\ZZ$ are subject to the restrictions
in Lemma \ref{lem:At most 4 beta}. Suppose that $D'\G\leq0$, which
is equivalent to
\[
\frac{te_{0}\a}{d_{0}}\leq\b_{1}.
\]
Using this relation in equation \eqref{eq:minus2}, we get 
\[
-1\leq2t\a^{2}-\left(\frac{te_{0}\a}{d_{0}}\right)^{2}-\sum_{i=2}^{16}\b_{i}^{2}.
\]
By using the relation \eqref{eq:Relationd0e0}, this is equivalent
to 
\begin{equation}
-1\leq-\frac{\a^{2}t}{d_{0}^{2}}-\sum_{i=2}^{16}\b_{i}^{2}\,\,\Longleftrightarrow\,\,\frac{\a^{2}t}{d_{0}^{2}}+\sum_{i=2}^{16}\b_{i}^{2}\leq1.\label{eq:LastEqua}
\end{equation}
Suppose that $\a$ is an integer. Then from Lemma \ref{lem:At most 4 beta}
and equation \eqref{eq:LastEqua}, we have either $\forall j\geq2,\,\b_{j}=0$
or $\exists k\geq2,\,\b_{k}=1$ and $\forall j\geq2,\,j\neq k,\,\b_{j}=0$.
In the first case $\G=\a L-\b_{1}A_{1}$ with $\b_{1}\in\ZZ$, and
from Lemma \ref{lem:treize}, either $\G=A_{1}$ or $\G=A_{1}'$.
Since $D'A_{1}=D'A_{1}'=4d_{0}>0$, this is impossible. In the second
case, $\sum_{i=2}^{16}\b_{i}^{2}=1$ implies $\a=0$ and $\G=A_{k}$
for $k\geq2$, and indeed $D'A_{k}=0$. It remains the case when $\a$
is an half-integer. Then three of the $\b_{i}$ with $i>1$ are equal
to $\frac{1}{2}$, the others are $0$, and $\b_{1}$ is in $\frac{1}{2}\ZZ\setminus\ZZ$.
Let $a,b\in\ZZ$ be the odd integers such that $\a=\frac{a}{2},\,\b_{1}=\frac{b}{2}.$
Equation \eqref{eq:minus2} becomes $b^{2}-2ta^{2}=1.$ By hypothesis,
the fundamental solution $(\a_{0},\b_{0})$ of that Pell-Fermat equation
$\a^{2}-2t\b^{2}=1$ is such that $\b_{0}$ is even. Then an easy
induction shows that every solution $(\a,\b)$ is also such that $\b$
is even. Hence that case is also impossible, and we obtain that $D'$
is nef with $D'\G=0$ for a $(-2)$-curve $\G$ if and only if $\G=A_{k}$
for $k\geq2$. 

\textit{Let us prove that the linear system $|D'|=|e_{0}L-2d_{0}A_{1}|$
is base point free.} Suppose that this is not the case. Then (see
\cite{SD}) there exist an elliptic curve $E$ and a $(-2)$-curve
$\G=\a L-\sum_{i=1}^{16}\b_{i}A_{i},\,\,\a,\b_{i}\in\frac{1}{2}\ZZ,$
such that $E\G=1$ and $D'=3E+\G$. One has $D'\G=(3E+\G)\G=1$ but
\[
D'\G=(e_{0}L-2d_{0}A_{1})(\a L-\sum_{i=1}^{16}\b_{i}A_{i})=2(2\a e_{0}t-2\b_{1}d_{0})\in2\ZZ.
\]
This is a contradiction, and we conclude that $|D'|$ is base-point
free.

\textit{Let us study the degree of the morphism defined by the linear
system $|D'|$}. By \cite{SD}, the morphism has degree $2$ if and
only if there exists en elliptic curve 
\[
E=\frac{a}{2}L-\sum_{i=1}^{16}\frac{b_{i}}{2}A_{i},\,\,a,b_{i}\in\ZZ
\]
such that $D'E=2$. Equality $D'E=2$ is equivalent to 
\begin{equation}
ate_{0}-b_{1}d_{0}=1.\label{eq:relationAB}
\end{equation}
Since $E$ is an elliptic curve, one has 
\begin{equation}
2ta^{2}=\sum_{i=1}^{16}b_{i}^{2}.\label{eq:ellCurve}
\end{equation}
Let us define $S=\sum_{i=2}^{16}b_{i}^{2}\in\NN$. The equations \eqref{eq:relationAB}
and \eqref{eq:ellCurve} give
\[
\left\{ \begin{array}{c}
2a^{2}t^{2}=\frac{2}{e_{0}^{2}}(1+b_{1}d_{0})^{2}\\
2t^{2}a^{2}=tb_{1}^{2}+tS
\end{array}\right.,
\]
therefore 
\[
2(1+b_{1}d_{0})^{2}=te_{0}^{2}b_{1}^{2}+te_{0}^{2}S,
\]
 which is equivalent to 
\begin{equation}
(e_{0}^{2}t-2d_{0}^{2})b_{1}^{2}-4d_{0}b_{1}+e_{0}^{2}tS-2=0.\label{eq:complique}
\end{equation}
Since $e_{0}^{2}t-2d_{0}^{2}=1$, the reduced discriminant of equation
\eqref{eq:complique} in $b_{1}$ is 
\[
\Delta=4d_{0}^{2}+2-e_{0}^{2}tS=4d_{0}^{2}+2-(1+2d_{0}^{2})S=(1+2d_{0}^{2})(2-S).
\]
Since $b_{1}$ is an integral solution, that implies $S\in\{0,1,2\}$.
Lemma \ref{lem:At most 4 beta} implies that cases $S=1$ and $S=2$
are not possible. The case $S=0$ is not possible either since $2(1+2d_{0}^{2})$
is not a square. We thus proved that the morphism defined by $|D'|$
has degree $1$ and that concludes the proof of Proposition \ref{prop:Degre4ModelWith15Nodes}.
\end{proof}
Using that model $Y\hookrightarrow\PP^{3}$, we can use the same
reasoning as in \cite[Proof of Theorem 19]{RS} (which was made for
the case $t=3$, with $d_{0}=e_{0}=1$) in order to prove that sub-case
\textit{$\frac{\alpha_{0}-1}{2d_{0}^{2}}=2$} is also impossible if
we suppose that the non-symplectic automorphism $f$ exists.

\subsection*{Case when $f$ is a symplectic involution.}

Suppose $f$ is a symplectic involution. It fixes $A_{1}+A_{1}'$,
it permutes the classes $A_{j}$ ($j>1$) by pairs, thus the number
$s$ of fixed $A_{j}$ is odd. A symplectic automorphism acts trivially
on the transcendental lattice $T_{X}$, which in our situation has
rank $5$. Therefore the trace of $f$ on $H^{2}(X,\ZZ)$ equals $5+\rk\NS X)^{f})\geq6+s>6$.
But the trace of a symplectic involution equals $6$ (see e.g. \cite[Section 1.2]{SvG}).
This is a contradiction, thus $f$ cannot have order $2$ and the
integer $m$ (such that the order of $f$ is $2^{m}$) is larger than
$1$.

\subsection*{Remaining cases}

We know that $f$ has order $2^{m}>1$. The automorphism $g=f^{2^{m-1}}$
has order $2$ and $g(A_{1})=A_{1},\,g(A_{1}')=A_{1}'$, thus $g(L)=L$.
There are curves $A_{i},\,i>1$ such that $f(A_{i})=A_{i}$ (say $s$
of such curves, $s$ is odd since $A_{1}$ is fixed) and the remaining
curves $A_{j}$ are permuted $2$ by $2$ (there are $s'=\frac{1}{2}(15-s)$
such pairs). Let $\mathcal{L}'$ be the sub-lattice generated by $L,A_{1}$
and the fix classes $A_{i}$, $A_{j}+g(A_{j})$. It is a finite index
sub-lattice of the fixed lattice $\NS X)^{g}$ and its discriminant
group is 
\[
\ZZ/4t\ZZ\times(\ZZ/2\ZZ)^{s+1}\times(\ZZ/4\ZZ)^{s'}.
\]
 By the same reasoning as before, the involution $g$ must be symplectic
as soon as $t>1$. However the trace of $g$ is $8+s>6$, thus $g$
cannot be symplectic either. Therefore we conclude that such an automorphism
$f$ does not exist, which conclude the proof of Theorem \ref{thm:no automorphisms}.
\begin{rem}
In the course of the proof of Theorem \ref{thm:no automorphisms},
using the divisor $D'$, we obtained a model of our K3 surface as
a quartic in $\PP^{3}$ with $15$ nodes, as soon as Equation \eqref{eq:Relationd0e0}
has a solution. This is the case for example for $t=3,9,11,19,27$...
\end{rem}

\subsection{When the negative Pell-Fermat equation has a solution}

Suppose that the negative Pell-Fermat equation $\l^{2}-2t\mu^{2}=-1$
has a solution and let $(d_{0},e_{0})$ be the fundamental solution.
Then $(\a_{0},\b_{0})=(1+2d_{0}^{2},2e_{0}d_{0})$ is the fundamental
solution of the Pell-Fermat equation $\l^{2}-2t\mu^{2}=1$, in particular
$\b_{0}$ is even. Moreover we have $A_{1}+A_{1}'=2d_{0}D'$ (we keep
the notation of the previous Section) with 
\[
D'=e_{0}L-d_{0}A_{1}
\]
such that $D'^{2}=2$. Let us prove that 
\begin{prop}
\label{prop:DoublePlaneModel}The divisor $D'$ is nef. We have $D'\G=0$
for a $(-2)$-curve $\G$ if and only if $\G=A_{j},$ for $j\in\{2,\dots,16\}$.
The linear system $|D'|$ is base point free. It defines a double
cover $\varphi:X\to\PP^{2}$ which contracts exactly the $15$ curves
$A_{j},j\geq2$ to $15$ singular points of a sextic curve which is
the union of $6$ lines. \\
The involution $\s$ defined by the double cover $\varphi$ exchanges
$A_{1}$ and $A_{1}'$ and the two Nikulin configurations $\mathcal{C},\mathcal{C}'$,
which therefore give the same Kummer structure. 
\end{prop}

\begin{proof}
\textit{Let us prove that $D'$ is nef}. Let $\G=\a L-\sum_{i=1}^{16}\b_{i}A_{i}$
be a $(-2)$-curve: $2t\a^{2}-\sum_{i=1}^{16}\b_{i}^{2}=-1$, where
$\a,\b_{i}\in\frac{1}{2}\ZZ$ are subject to the restrictions in Lemma
\ref{lem:At most 4 beta}. We suppose that $D'\G\leq0$. This is equivalent
to:
\[
D'\G=(e_{0}L-d_{0}A_{1})(\a L-\sum_{i=1}^{16}\b_{i}A_{i})=4t\a e_{0}-2\b_{1}d_{0}\leq0.
\]
Then $\b_{1}\geq\frac{2t\a e_{0}}{d_{0}}$, thus $-\left(\frac{2t\a e_{0}}{d_{0}}\right)^{2}\geq-\b_{1}^{2}$.
From the relation $2t\a^{2}+1-\b_{1}^{2}=S$, where $S=\sum_{i=2}^{16}\b_{i}^{2}$,
we get
\[
S\leq1+2t\a^{2}-\left(\frac{2t\a e_{0}}{d_{0}}\right)^{2}.
\]
By using the relation $d_{0}^{2}-2te_{0}^{2}=-1$, we obtain
\begin{equation}
S+\frac{2t\a^{2}}{d_{0}^{2}}\leq1.\label{eq:analogLast}
\end{equation}
We can then follow the same proof as for the divisor in Proposition
\ref{prop:Degre4ModelWith15Nodes}, and we conclude that $D'$ is
nef, with $D'\G=0$ for a $(-2)$-curve $\G$ if and only if $\G=A_{k}$
for $k\in\{2,\dots,16\}$.

\textit{Let us prove that the linear system $|D'|=|e_{0}L-d_{0}A_{1}|$
is base point free.} Suppose that this is not the case. Then (see
\cite{SD}) there exist an elliptic curve $E$ and a $(-2)$-curve
$\G$ such that $D'=2E+\G$ and $E\G=1$. Since we have $D'\G=(2E+\G)\G=0$,
there exists $k\geq2$ such that $\G=A_{k}$. Thus $D'=2E+A_{k}$,
and then $e_{0}L-d_{0}A_{1}-A_{k}=2E$, which is impossible by Lemma
\ref{lem:At most 4 beta}. Therefore $|D'|$ is base point free and
defines a double cover of the plane that contracts the $15$ disjoint
$(-2)$-curves $A_{j},j\geq2$ to points. Since $A_{1}+A_{1}'=2d_{0}D'$,
the divisor $A_{1}+A_{1}'$ is the pull-back of a plane rational cuspidal
curve of degree $2d_{0}$, and the involution exchanges the two curves
$A_{1},A_{1}'$. It is well-known that a sextic curve with $15$ singular
points is the union of $6$ lines in general position. 
\end{proof}

\section{Further examples}

\subsection{\label{subsec:An-example-whent odd} An example of a Nikulin configuration
when $\protect\b_{0}$ is odd}

Let us study the $t=4$ case. Then the fundamental solution $(\a_{0},\b_{0})$
equals $(3,1)$. This is the first case with $\beta_{0}$ odd (see
Table in the Appendix). We have 
\[
A_{1}'=L-3A_{1},\,\,L'=3L-8A_{1}
\]
and we already know that $A_{1}'$ is not irreducible. In order to
understand better what is happening, let us define
\[
\begin{array}{c}
A_{1}''=\frac{1}{2}(L-3A_{1}-A_{2}-A_{3}-A_{4})\\
A_{2}''=\frac{1}{2}(L-A_{1}-3A_{2}-A_{3}-A_{4})\\
A_{3}''=\frac{1}{2}(L-A_{1}-A_{2}-3A_{3}-A_{4})\\
A_{4}''=\frac{1}{2}(L-A_{1}-A_{2}-A_{3}-3A_{4})
\end{array},
\]
where the classes $A_{2},A_{3},A_{4}$ are chosen so that the classes
$A_{j}''$ exists in $\NS X)$ (which is possible by \cite[Theorem 2.7]{GS},
since $t=0\text{ mod }2$). We compute that these are $(-2)$-classes
i.e. $A_{j}''^{2}=-2$. Moreover we have 
\[
A_{1}'=2A_{1}''+A_{2}+A_{3}+A_{4}\text{ and }A_{1}''L'=0,
\]
(but $A_{i}''L'\not=0$, for $i=2,3,4$). Let us also define 
\[
L_{1}=3L-4(A_{1}+A_{2}+A_{3}+A_{4}).
\]
We remark that $L_{1}^{2}=L^{2}=16$, $L_{1}A_{i}''=0$ and $A_{i}''A_{j}''=0$
for $i\neq j$ in $\{1,2,3,4\}$.
\begin{lem}
The class $L_{1}$ is big and nef. Moreover if $\G$ is an effective
$(-2)$-class, we have $L_{1}\G\geq0$ and $L_{1}\G=0$ if and only
if $\G$ is one of the classes $A_{1}'',...,A_{4}'',A_{5},\dots,A_{16}$. 
\end{lem}

\begin{proof}
Let 
\[
\Gamma=aL-\sum_{i=1}^{16}b_{i}A_{i}
\]
be an effective $(-2)$-class (thus $\sum b_{i}^{2}-8a^{2}=1$). One
has 
\[
L_{1}\G\leq0
\]
 if and only if 
\[
6a\leq b_{1}+b_{2}+b_{3}+b_{4}.
\]
Suppose $\G\not\in\{A_{5},\dots,A_{16}\}$. Then $a>0,\,b_{i}\geq0$
and equation $L_{1}\G\leq0$ is equivalent to 
\[
a^{2}\leq\frac{1}{36}(b_{1}+\dots+b_{4})^{2}.
\]
Using 
\[
(b_{1}+\dots+b_{4})^{2}\leq4(b_{1}^{2}+\dots+b_{4}^{2})\leq4(b_{1}^{2}+\dots+b_{16}^{2})
\]
we get 
\[
a^{2}\leq\frac{1}{9}(b_{1}^{2}+\dots+b_{16}^{2})
\]
and since $\sum b_{i}^{2}=8a^{2}+1$, we have
\[
a^{2}\leq\frac{1}{9}(8a^{2}+1),
\]
thus $a^{2}\leq1$ and $a\in\{\frac{1}{2},1\}$. Suppose that $a=\frac{1}{2}$.
Then $\sum_{i=1}^{i=16}b_{i}^{2}=3$ and either there are $12$ $b_{i}$'s
equal to $\frac{1}{2}$ or (up to permutation of the indices) $b_{1}=\frac{3}{2}$,
$b_{2}=b_{3}=b_{4}=\frac{1}{2}$. The first case is impossible since
$3=6a>b_{1}+b_{2}+b_{3}+b_{4}$. The second case corresponds to $A_{1}'',\dots,A_{4}''$,
and then $L_{1}A_{j}''=0$. \\
It remains to study the case $a=1$, then $\sum b_{i}^{2}=9$ and
\[
6\leq b_{1}+b_{2}+b_{3}+b_{4}.
\]
That implies $b_{i}\leq\frac{5}{2}$. Up to permutation we can suppose
that the largest $b_{i}$ with $i\in\{1,2,3,4\}$ is $b_{1}$. Suppose
$b_{1}=\frac{5}{2}$, then $\sum_{i\geq2}b_{i}^{2}=\frac{11}{4}$
and $b_{2}+b_{3}+b_{4}\geq\frac{7}{2}$. One can suppose that $b_{2}$
is the largest among $b_{2},b_{3},b_{4}$, then there are two cases
: $b_{2}=2$ or $b_{2}=\frac{3}{2}$. The first case is impossible
since one would obtain $\sum_{i\geq2}b_{i}^{2}>\frac{11}{4}$. Suppose
$b_{2}=\frac{3}{2}$, then $b_{3}^{2}+b_{4}^{2}=\frac{1}{2}$ and
$b_{3}+b_{4}\geq2$, but this is also impossible. \\
Suppose that $b_{1}=2$. Then $\sum_{i\geq2}b_{i}^{2}=5$ and $b_{2}+b_{3}+b_{4}\geq4$.
The largest $b_{i}$ among $b_{2},b_{3},b_{4}$ (say it is $b_{2}$)
is $2$ or $\frac{3}{2}$. If $b_{2}=2$, then $b_{3}^{2}+b_{4}^{2}=1$
and $b_{3}+b_{4}\geq2$, which is impossible. If $b_{2}=\frac{3}{2}$,
then $b_{3}+b_{4}\geq\frac{5}{2}$ and $b_{3}^{2}+b_{4}^{2}=\frac{11}{4},$thus
$b_{3}=\frac{3}{2}$ and $b_{4}\geq1$ gets a contradiction. \\
It remains $b_{1}=\frac{3}{2}$, but then $b_{2}=b_{3}=b_{4}=\frac{3}{2}$.
That implies $b_{j}=0$ for $j\geq5$. But $L-\frac{3}{2}(A_{1}+A_{2}+A_{3}+A_{4})$
is not in the Néron-Severi group of the surface (see Lemma \ref{lem:At most 4 beta}).
\\
We thus proved that the only effective $(-2)$-classes $\G$ such
that $L_{1}\G\leq0$ are $A_{1}'',\dots,A_{4}'',A_{5},\dots,A_{16}$
and moreover $L_{1}\G=0$ for these classes. Thus $L_{1}$ is nef
and big. 
\end{proof}
As before, one can prove that the linear system $3L_{1}$ define a
degree $1$ morphism which contracts $A_{1}'',\dots,A_{4}'',A_{5},\dots,A_{16}$
onto singularities. Since we assume that the Kummer surface is generic,
it has Picard number $17$ and we conclude that the divisors $A_{1}'',\dots,A_{4}''$
are irreducible. Therefore:
\begin{cor}
The $16$ $(-2)$-curves $A_{1}'',\dots,A_{4}'',A_{5},\dots,A_{16}$
form a Nikulin configuration $\mathcal{C}'$ on the K3 surface $X$.
The $(-2)$-class $A_{1}'$ is not irreducible and $A_{1}'=2A_{1}''+A_{2}+A_{3}+A_{4}$.
\end{cor}

\begin{rem}
i) One can check that the class $L'$ is big and nef; the image of
$X$ by the linear system $|3L'|$ is a surface with $12$ nodal singularities
and one $D_{4}$ singularity obtained by contracting $A_{1}'',A_{2},A_{3},A_{4}$.
\\
ii) We do not know yet if $\mathcal{C}'$ is another Kummer structure
on the Kummer surface $X$, we intend to study that problem in a forthcoming
paper.
\end{rem}

\subsection{\label{subsec:An-example-with}An example of a Nikulin configuration
when $2t$ is a square }

Let us consider the case $t=2$ i.e. $A$ is a $(1,2)$-polarized
abelian surface. Then $2t$ is a square and the method in Section
\ref{subsec:The--even} does not apply. We start by recalling the
following 
\begin{rem}
Since $t$ is even, by \cite[Theorem 2.7 and Remark 2.10]{GS}, up
to re-labelling the $16$ curves $(-2)$-curves $A_{j}$, we can suppose
that the classes 
\[
\begin{array}{cc}
F_{1}=\frac{1}{2}(L-A_{1}-A_{2}-A_{3}-A_{4}),\,\,\, & F_{2}=\frac{1}{2}(L-A_{5}-A_{6}-A_{7}-A_{8}),\,\,\,\\
F_{3}=\frac{1}{2}(L-A_{9}-A_{10}-A_{11}-A_{12}), & F_{4}=\frac{1}{2}(L-A_{13}-A_{14}-A_{15}-A_{16})
\end{array}
\]
are contained in $\NS X)$. For $j\in\{1,...,4\}$, we define 
\[
B_{j}=F_{2}-A_{j}
\]
and for $j\in\{5,...,8\}$, we define 
\[
B_{j}=F_{1}-A_{j}.
\]
These are $(-2)$-classes; they are effective since $LB_{j}>0$. We
check moreover that 
\[
B_{j}B_{k}=-2\delta_{jk},
\]
where $\delta_{jk}$ is the Kronecker symbol. Let us prove the following
result:
\end{rem}

\begin{prop}
\label{prop:NikulinConfFor2tA Square}The classes $B_{1},\dots,B_{8},A_{9},\dots,A_{16}$
are $16$ disjoint $(-2)$-curves.
\end{prop}

\begin{proof}
We have $B_{k}A_{j}=0$ for $k\in\{1,\dots,8\}$ and $j\in\{9,\dots,16\}$.
It remains to prove that $B_{1},\dots,B_{8}$ are irreducible. We
compute that $F_{1}^{2}=0=F_{2}^{2}$, $F_{1}F_{2}=2$. Since $LF_{1}=4,$
the divisor $F_{1}$ is effective. The fact that $F_{1}$ is nef can
be found in \cite[Proposition 4.6]{GS}, but for completeness, let
us prove it here. Suppose that $F_{1}$ is not nef. Then there exist
a $(-2)$-curve $\G$ such that $F_{1}\G=-a<0$. By \cite[Remark 8.2.13]{Huyb},
the divisor $E=F_{1}-a\G$ is effective, moreover 
\[
E^{2}=0,\,E\G=a,\text{ (and }F_{1}=E+a\G\text{).}
\]
 Let us write $E=\a L-\sum\b_{i}A_{i}$ with $\a\in\frac{1}{2}\ZZ,\,\text{and }\a>0$
since $E$ is effective and the lattice $L^{\perp}$ is negative definite.
Since $L$ is nef and
\[
4=F_{1}L=8\a+aL\G
\]
with $a>0$, that forces $\a=\frac{1}{2}$ and $L\G=0$. Thus $\G$
is one of the curves $A_{j}$. But for these curves $F_{1}A_{j}\in\{0,1\}$
is not negative, a contradiction. Therefore $F_{1}$ is nef and the
linear system $|F_{1}|$ has no base points. We prove similarly that
all the linear system $|F_{k}|,\,k\in\{1,2,3,4\}$ have no base points,
and then defines a fibration $\psi_{k}:X\to\PP^{1}$. Since $F_{1}A_{j}=1$
for $j\in\{1,...,4\}$, the fibration $\psi_{1}$ has connected fibers.
By the same kind of argument so is $\psi_{2}$. For $k\in\{5,...,16\}$,
let us define 
\[
C_{k}=F_{1}-A_{k}
\]
(so that in fact $B_{k}=C_{k}$ for $k\in\{5,...,8\}$). The divisor
$C_{k}$ is an effective $(-2)$-class and the $12$ divisors 
\[
C_{k}+A_{k},\,k\in\{5,\dots,16\}
\]
are distinct singular fibers of $\psi_{1}$, with $A_{k}C_{k}=2$.
 We now use \cite[Proposition 11.4, Chapter III]{BHPVdV}: the Euler
characteristic of $X$ (equal to $24$) is the sum $\sum_{s}e(f_{s})$
of the Euler numbers of all the singular fibers. By the Kodaira classification
of singular fibers of elliptic fibrations (see e.g. \cite[Table 3, Chapter V, Section 7]{BHPVdV}),
the reducible fibers $f_{s}=C_{k}+A_{k}$ for $k\geq5$ satisfy $e(f_{s})\geq2$.
Moreover, by the above cited Table, a singular fiber $f_{s}$ containing
a smooth rational curve satisfies $e(f_{s})=2$ if and only if it
is the union of two $(-2)$-curves $D_{1},D_{2}$ with $D_{1}D_{2}=2$
and meeting transversally. Computing the Euler characteristic of $X$,
we see that necessarily $e(C_{k}+A_{k})=2$, for $k\in\{5,\dots,16\}$
and therefore the curves $B_{k}=C_{k}$, $k\in\{5,...,8\}$ are irreducible
$(-2)$-curves. We proceed in a similar way with $\psi_{2}$ for the
curves $B_{k}$ with $k\in\{1,...,4\}$, and we thus obtain the claimed
result.
\end{proof}

\begin{rem}
i) By the Proposition \ref{prop:NikulinConfFor2tA Square}, we see
that the elliptic fibration defined by $F_{1}$ contains $12$ fibers
of type $I_{2}$. By general results on elliptic K3 surfaces, the
rank $\rho$ of the Néron-Severi group is $14=12+2$ plus the rank
of the Mordell-Weil group, which is the group generated by the zero
section (we can take $A_{1}$ as zero section) and the sections of
infinite order. Since we know that $\rho=17$ we get that the rank
of the Mordell-Weil group is three. That group contains the disjoint
sections $A_{2}$, $A_{3}$ and $A_{4}$. The remark is similar for
the fibration defined by $F_{2}$. \\
ii) On the K3 surface $X$ we have two Nikulin configurations 
\[
\mathcal{C}=\sum_{i=1}^{16}A_{i},\,\mathcal{C}'=\sum_{i=1}^{8}B_{i}+\sum_{i=9}^{16}A_{i}.
\]
We do not know if these configurations define two Kummer structures
on $X$. We intend to come back on the subject later. \\
iii) It is also possible to check that the divisor $L'=3L-2(A_{1}+\dots+A_{8})$
is big and nef, $L'^{2}=8$ and $L'\G=0$ for an effective $(-2)$-class
$\G$ if and only if $\G$ is in $\{B_{1},\dots,B_{8},A_{9},\dots,A_{16}\}$. 

\end{rem}

\section*{Appendix }

\subsection*{\label{subsec:Why-it-was}Why it was natural to study the case $t=\frac{1}{2}k(k+1)$
in the paper \cite{RS}.}

Since $\a^{2}=1+2t\b^{2}$, the integer $\a$ is odd. Let $k\in\NN$
be such that $\a=2k+1$ (then one has $A_{1}A_{1}'=4k+2$). The integer
$\b$ is then solution of the equation
\[
(2k+1)^{2}-2t\b^{2}=1,
\]
which is equivalent to
\[
t\b^{2}=2k(k+1).
\]
Then 
\[
a=2t\b,\,b=2k+1
\]
are solutions of the three conditions in \eqref{eq:tree conditions}.
Since $a^{2}=2t(b^{2}-1)$, one gets 
\begin{equation}
a^{2}=2t\cdot4k(k+1).\label{eq:square}
\end{equation}
Thus $2t\cdot4k(k+1)$ must be the square of an integer and it is
therefore natural to define
\[
t=\frac{1}{2}k(k+1).
\]
Then one computes easily that $a=2k(k+1)$ and $\b=2$. Then one has
\[
GCD(\b,\a_{0}-1)=GCD(2,2k)=2,
\]
thus as soon as $\a_{0}>5,$ i.e. $k>2$, one can apply Theorem \ref{thm:no automorphisms}.
That were the cases we studied in \cite{RS}.

\subsection*{A table}

We resume in the following table the fundamental solutions $(\a{}_{0},\b{}_{0})$
of the Pell-Fermat equation $\alpha^{2}-2t\beta^{2}=1$ for $2t\leq60$.
Recall that there are non-trivial solutions if and only if $2t$ is
not a square. Observe that when $2t=k(k+1)$ the minimal solution
is $(2k+1,2)$, these correspond to Nikulin configurations studied
in the paper \cite{RS}, we put a $*$ close to these cases. We put
a box around the cases with $\beta_{0}$ odd, and a prime $'$ when
$\b_{0}$ is even but such that the negative Pell-Fermat equation
has a solution: these cases are left out in this paper. 

\begin{center}
\begin{table}[ht] 
\caption{Fundamental solutions of the Pell-Fermat equations}\label{tabelle_Pell} 
	\centering 
\begin{tabular} {c|c|c|c|c|c|c|c|c|c|c|c|c|c|c|c}  

\hline 
$2t$&2*&4&6*&\boxed{8}&10'&12*&14&16&18&20*&22&\boxed{24}&26'&28&30*\\ 
	\hline 
$\alpha_0$&3&-&5&3&19&7&15&-&17&9&197&5&51&127&11\\	 	
\hline 
$\beta_0$&2&-&2&1&6&2&4&-&4&2&42&1&10&24&2\\ 

\hline 
\hline 
\hline
$2t$&\boxed{32}&34&36&38&\boxed{40}&42*&44&46&\boxed{48}&50'&52&54&56*&58'&60\\ 
\hline 
$\alpha_0$&17&35&-&37&19&13&199&24335&7&99&649&485&15&19603&31\\ 
\hline 
$\beta_0$&3&6&-&6&3&2&30&3588&1&14&90&66&2&2574&4\\ 
\hline 
\end{tabular}  
\end{table} 
\end{center}

 \vspace{5mm}
\noindent Xavier Roulleau,
\\Aix-Marseille Universit\'e, CNRS, Centrale Marseille,
\\I2M UMR 7373,  
\\13453 Marseille,
\\France
\\ \email{Xavier.Roulleau@univ-amu.fr}
\vspace{0.3cm} \\
%
\noindent 
\author{Alessandra Sarti} \\
\address{	Universit\'e de Poitiers\\
Laboratoire de Math\'ematiques et Applications,\\
UMR 7348 du CNRS, \\
TSA 61125	
\\11 bd Marie et Pierre Curie, 		\\
	86073 POITIERS Cedex 9, \\France}\\ 
\email{Alessandra.Sarti@math.univ-poitiers.fr} \\
\begin{verbatim}
http://www.i2m.univ-amu.fr/perso/xavier.roulleau/Site_Pro/Bienvenue.html
http://www-math.sp2mi.univ-poitiers.fr/~sarti/

\end{verbatim}

\end{document}